\documentclass{amsart}[11pt]
\usepackage{amsmath, amsthm, amssymb}
\usepackage{graphicx}
\usepackage{eepic}
\usepackage{pstricks}
\newtheorem{thm}{Theorem}[section]

\newtheorem{prop}[thm]{Proposition}

\newtheorem{cor}[thm]{Corollary}

\newtheorem{lem}[thm]{Lemma}

\newtheorem{fact}[thm]{Fact}

\numberwithin{equation}{section}

\def\Ini{\mathrm{Ini}}
\def\T{\mathrm{T}}
\def\Int{\mathrm{Int}}

\def\Pa{{\mathrm P}}

\def\S{{\mathcal S}}
\def\PS{{\mathcal PS}}
\def\sig{{\sigma}}
\def\a{\mathbf{a}}
\def\b{\mathbf{b}}
\def\c{\mathbf{c}}

\def\x{\mathbf{x}}
\def\y{\mathbf{y}}

\def\u{\mathbf{u}}

\def\pc{\textrm{pc}}
\def\ep{\mathbf{\epsilon}}
\def\P{{\mathcal P}}

\def\C{{\mathcal C}}
\def\1{{\bf 1}}
\def\0{{\bf 0}}

\begin{document}

\title{The most frequent peak set of a random permutation}

\author[A. Kasraoui]{Anisse Kasraoui$^{*}$}
\address{Fakult\"at f\"ur Mathematik, Universit\"at Wien,
Nordbergstrasse 15,
A-1090 Vienna,
Austria}
\email{anisse.kasraoui@univie.ac.at}
\thanks{$^{*}$ Research supported by the grant S9607-N13 from Austrian Science Foundation FWF
 in the framework of the National Research Network  ``Analytic Combinatorics and Probabilistic Number theory".}
 

\begin{abstract}
Given a subset $S\subseteq\mathbb{P}$, let $\Pa(S;n)$ be the number of permutations in the symmetric group of $\{1,2,\ldots,n\}$
that have peak set $S$. We prove a recent conjecture due to Billey, Burdzy and Sagan, which determines 
the sets that maximize $\Pa(S;n)$, where $S$ ranges over all subsets of $\{1,2,\ldots,n\}$.
\end{abstract}
\maketitle

\section{Introduction}

Given a permutation $\sig=\sig_1\,\sig_2\,\cdots \,\sig_n$ in
$\S_n$, the symmetric group of the set $[n]=\{1,2,\ldots,n\}$, the
up-down sequence of $\sig$ is the word
$F(\sig)=z_1\,z_2\,\cdots\,z_{n-1}$ with letters in $\{+,-\}^{n-1}$
defined by $z_i=-$ if $i$ is a descent of $\sig$ (i.e.,
$\sig_i>\sig_{i+1}$) and $z_i=+$  if $i$ is an ascent of $\sig$
(i.e., $\sig_i<\sig_{i+1}$).  The enumeration of permutations with a
given up-down sequence is a traditional topic of classical
combinatorics which has received considerable interest since the end
of the 19th century (see
e.g.~\cite{Mac,Bruijn,Carlitz,Foulkes,Niven,Viennot}). A classical
result in this topic that has been proved several times during the
last century
 (see e.g.~\cite{Bruijn,Niven,Viennot} but the list is not exhaustive)
 is that the maximum of $\beta(w)$, the number of permutations of $[n]$ with up-down sequence $w$,
where $w$ ranges over all up-down sequences of length $n-1$, occurs
when $w$ is an alternating sequence, i.e. $w=+-+-+-\cdots$ or
$w=-+-+-+\cdots$. A refinement of this maximization problem which
consists in maximizing $\beta(w)$ where $w$ ranges over all up-down
sequences of length $n$ with a fixed number of alternating runs, has
been conjectured by Gessel and was solved by Ehrenborg and
Mahajan~\cite{Ehrenborg}.

A related maximization problem has recently arisen  in the work on the peak statistic by
Billey, Burdzy and Sagan~\cite{Billey}. Recall that, given a
permutation $\sig=\sig_1\,\sig_2\,\cdots \,\sig_n$ in~$\S_n$, an
index~$i$ is called a peak of~$\sigma$ if $\sig_{i-1} < \sig_{i}
> \sig_{i+1}$. The peak set is defined to be
\begin{align*}
\PS(\sig)&=\{i\,:\, \text{$i$ is a peak of $\sig$}\},
\end{align*}
and given a set $S$ of integers, we set
\begin{align*}
\P(S;n)&=\{\sig\in\S_n\,:\, \PS(\sig)=S\},
\end{align*}
and $\Pa(S;n)=\#\P(S;n)$. For instance,  if 
$\sig=2\,6\,5\,1\,4\,3$, we have $\PS(\sig)=\{2,5\}$ and thus $\sigma\in\P(\{2,5\};6)$.
The numbers $\Pa(S;n)$ have been the center of
the work of Billey, Burdzy and Sagan~\cite{Billey} and a
conjecture on the sets $S$ that maximize $\Pa(S;n)$ among all subsets
of $[n]$ has been proposed. The main purpose of the present paper is
to solve this conjecture. To present their conjecture, it is more
convenient to use the notion of peak-composition.

 Suppose that the peak set of $\sig\in \S_n$ is
$\PS(\sig)=\{i_1<i_2<\cdots<i_k\}$. Then, the composition
$\pc(\sig)=(c_1,c_2,\ldots,c_{k+1})$ of $n$ defined by
$c_j=i_j-i_{j-1}$ with $i_0=0$ and $i_{k+1}=n$ is called the
peak-composition of $\sig$. If $\c$ is a composition of $[n]$, we
set
\begin{align*}
\P(\c)&=\{\sig\in\S_n\,:\, \pc(\sig)=\c\},
\end{align*}
and $\Pa(\c)=\#\P(\c)$.  For instance, if $\sig=2\,6\,5\,1\,4\,3$, we have $\pc(\sig)=\{2,3,1\}$ and thus $\sigma\in\P(2,3,1)$.
Note that if $S=\{i_1<i_2<\cdots<i_k\}$ is a
subset of $[n]$ then $\P(S;n)=\P(\c)$ where $\c$ is defined as
above. For instance, we have $\P(\{2,5\};6)=\P(2,3,1)$.
The size of a composition is the sum of its parts. In the
sequel, we say that a composition $\c$ is maximal if $\Pa(\b)\leq
\Pa(\c)$ for any composition $\b$ having the same size as $\c$. The
following result has been conjectured by Billey, Burdzy and
Sagan.

\begin{thm}\label{thm:main}
For $n\geq1$,  let $\C^*(n)$ denote the set of maximal compositions
of $n$. For $\ell \geq 2$, we have
\begin{itemize}
\item $\C^*(3\ell)=\{(3^{\ell}),(4,3^{\ell-2},2)\}$,
\item $\C^*(3\ell+1)=\{(3^{s},2,3^{t},2)\;:\;s\geq1,\; t\geq0,\;s+t=\ell-1\}$,
\item $\C^*(3\ell+2)=\{(3^{\ell},2)\}$.
\end{itemize}
Equivalently, for $n\geq 6$, the sets that maximize $\Pa(S;n)$ are
\begin{itemize}
\item if $n\equiv0\pmod{3}$, $\{3,6,9,\ldots\}\cap \{1,2,\ldots n-1\}$ and $\{4,7,10,\ldots\}\cap \{1,2,\ldots n-1\}$,
\item if $n\equiv1\pmod{3}$, $\{3,6,9,\ldots,3s,3s+2,3s+5,\ldots\}\cap \{1,2,\ldots n-1\}$ with $1\leq s \leq \lfloor\tfrac{n}{3}\rfloor-1$,
\item if $n\equiv2\pmod{3}$, $\{3,6,9,\ldots\}\cap \{1,2,\ldots n-1\}$.
\end{itemize}
\end{thm}
In the above theorem, exponent just indicates iteration. For
instance, $(3^{2},4)$ is for $(3,3,4)$.
 We will also determine the cardinality of $\P(\c)$ when $\c$ is a maximal composition. The expressions obtained
are extremely simple.
\begin{thm}\label{thm:Maximalclasses}
 Suppose $n\geq6$ and $\c\in\C^*(n)$. Set $\ell=\lfloor\tfrac{n}{3}\rfloor$. Then, we have,
  \begin{itemize}
\item if $n\equiv0\pmod{3}$, $\Pa(\c)=\tfrac{1}{5}3^{2-\ell}\,n!$,
\item if $n\equiv1\pmod{3}$, $\Pa(\c)=\tfrac{2}{5} 3^{1-\ell}\,n!$,
\item if $n\equiv2\pmod{3}$, $\Pa(\c)=3^{-\ell}\,n!$.
\end{itemize}
\end{thm}

Our proof of Theorem~\ref{thm:main} is essentially  based on some comparison lemmas presented in Section~2 
and a counting formula for the number $\Pa(\c)$ which is efficient when $\c$
contains only short patterns of elements distinct from 3. 
 In Section~3, we show that a maximal composition can not have a part greater than $4$.
In Section~4, we provide further patterns that a maximal composition must avoid. 
In Section~5, we preset our counting formula from which we deduce Theorem~\ref{thm:Maximalclasses}.
Finally, the proof of Theorem~\ref{thm:main} is completed in Section~6.

\section{Comparison Lemmas}
In this section, we provide some comparison lemmas. We first
set up some additional terminology. The concatenation of two compositions $\c_1$ and $\c_2$
is denoted by $\c_1\oplus\c_2$. For instance,
$(4,3,2)\oplus (2,3)=(4,3,2,2,3)$. 
We assume the existence of an empty (with 0 part) composition $\ep$
such that $\c\oplus\ep=\ep\oplus\c=\c$.
 We say that a composition $\c$
is admissible if $\P(\c)\neq \emptyset$. It is not hard to prove  the
following result.
\begin{fact}\label{fact:admissible}
A composition $\c=(c_1,\ldots,c_k)$ is admissible if and only if
$k=1$ or $c_1,\ldots,c_{k-1}\geq 2$ and $c_k\geq 1$.
\end{fact}
Given a composition $\c=(c_1,c_2,\ldots,c_k)$, we set $r'
\c=\c$ if $k=1$ and $r'\c=(c_{k}+1,c_{k-1},\ldots,c_2,c_{1}-1)$ if $k\geq 2$.
Considering the reverse of a permutation, we
obtain the following result.
\begin{fact}\label{fact:rev-compo}
For any composition $\c$, we have $\Pa(\c)=\Pa(r'\c)$. In particular,
$\c$ is maximal if and only if $r'\c$ is maximal.
\end{fact}

  For $a,b\in\mathbb{P}$ and  $\c$ a composition
of an integer $n\geq 1$, we let
\begin{align*}
Ini_{a,b}(\c)&
=\{\sigma\in\P(\c) \,:\, a=\sig_1,\;\sig_{n-1}<\sig_n=b \},\\
Int_{a,b}(\c)&= \{\sigma\in\P(\c) \,:\,
a=\sig_1>\sig_2,\;\sig_{n-1}<\sig_n=b \},
\end{align*}
and $Ini_{\cdot,b}(\c)=\bigcup_{a=1}^n Ini_{a,b}(\c)$. We set
$\Ini_{a,b}={\#}Ini_{a,b}$, $\Ini_{\cdot,b}={\#} Ini_{\cdot,b}$ and $\Int_{a,b}={\#}
Int_{a,b}$. We also let $1+\c$ denote the composition of $n+1$
obtained from $\c$ by increasing the first part of $\c$ by 1. So, if
$\c=(c_1,\ldots,c_k)$, $1+\c=(c_1+1,c_2,\ldots,c_k)$. For instance,
$1+(2,4,3,2)=(3,4,3,2)$.

 If ${\x}=(x_1,x_2,\ldots,x_k)$
and ${\y}=(y_1,y_2,\ldots,y_k)$ are two sequences of the same length, we
write ${\x}<{\y}$ if $x_i\leq y_i$ for $1\leq i\leq k$ and $\x\neq
\y$.

\begin{prop}\label{prop:comparison2}
Let $\c$ be a composition of an integer $n\geq1$.
Suppose that there exists a composition $\c'$ of $n$ such that

(1) $\Int_{a,b}(1+\c)\leq \Int_{a,b}(1+\c')$ for all
$a,b\in[n+1]$ and there are two indices $u,v$ such that
$\Int_{u,v}(1+\c)< \Int_{u,v}(1+\c')$: then, for any nonempty
compositions $\a$ and $\b$ such that $\a\oplus\c\oplus\b$ is
admissible, $$
\Pa(\a\oplus\c\oplus\b)<\Pa(\a\oplus\c'\oplus\b),
$$
and thus, $\a\oplus\c\oplus\b$ is not maximal.

(2)$\big(\Ini_{\cdot,b}(\c)\big)_{1\leq b\leq n}<
\big(\Ini_{\cdot,b}(\c')\big)_{1\leq b\leq n}$: then, for
any nonempty composition $\b$ such that $\c\oplus\b$ is admissible,
$$
\Pa(\c\oplus\b)<\Pa(\c'\oplus\b),
$$
and thus, $\c\oplus\b$ is not maximal.

\end{prop}

 To simplify the readability of the proof of the above result,
 we first state a preliminary lemma.
In the sequel, a permutation $\tau=\tau_1\cdots\tau_n$ of a set
$S\subset\mathbb{P}$ with cardinality $n$ is said to be
order-isomorphic to a permutation $\sigma=\sig_1\cdots\sig_n$ in
$\S_n$ if for $1\leq i<j\leq n$, $\tau_i<\tau_j$ is equivalent to
$\sigma_i<\sigma_j$.

\begin{lem}\label{lem:comparison}
For $i=1,2,3$, let $\c^{(i)}$ be a  composition of a
positive integer $n_i$ such that
$\c^{(1)}\oplus\c^{(2)}\oplus\c^{(3)}$ is admissible. Suppose
$\Int_{u,v}(1+\c^{(2)})>0$, then for each permutation $\tau\in
Int_{u,v}(1+\c^{(2)})$ there exists a permutation $\sig$ in
$\P(\c^{(1)}\oplus\c^{(2)}\oplus\c^{(3)})$ such that
$\sigma_{n_1}\sigma_{n_1+1}\cdots\sigma_{n_1+n_2}$  is order
isomorphic to $\tau$.
\end{lem}

\begin{proof} We first choose two permutations $\gamma\in
Ini_{\cdot,n_1}(\c^{(1)})$ and $\beta\in
Ini_{\cdot,n_3+1}\big(r'(1+\c^{(3)})\big)$ (this choice is always
possible since, as it is easily seen,
$Int_{\cdot,n}(\c)\neq\emptyset$ for any composition $\c$ of $n$).
Then set $\sig_1,\ldots,\sig_{n_1-1}=\gamma_1,\ldots,\gamma_{n_1-1}$
(resp.,
$\sig_{n_1+n_2+n_3},\sig_{n_1+n_2+n_3-1},\ldots,\sig_{n_1+n_2+1}=\beta_1+n_1-1,\ldots,\beta_{n_3}+n_1-1$).
We then set
$\sig_{n_1},\sig_{n_1+1},\ldots,\sig_{n_1+n_2}=\tau_{1}+n_1+n_3-1,\tau_2+n_1+n_3-1,\ldots,\tau_{n_2+1}+n_1+n_3-1$.
It is easily checked that the permutation
$\sigma=\sig_1\cdots\sig_{n_1+n_2+n_3}$ is in
$\P(\c^{(1)}\oplus\c^{(2)}\oplus\c^{(3)})$ and
$\sigma_{n_1}\sigma_{n_1+1}\cdots\sigma_{n_1+n_2}$  is order
isomorphic to $\tau$.
 \end{proof}

\noindent\emph{Proof of Proposition~\ref{prop:comparison2}.} (1) Let $\a$ and $\b$ be two compositions 
of positive integers $r$ and $s$ such that $\a\oplus\c\oplus\b$ is
admissible. 
By our hypothesis,  we can consider a family of injections
$\big(\phi_{c,d}\big)_{1\leq c,d\leq n+1}$,
$\phi_{c,d}:Int_{c,d}(1+\c)\mapsto Int_{c,d}(1+\c')$, such that
$\phi_{u,v}$ is not surjective. Then, consider the function $\Gamma$
which associates to a permutation $\sig\in\P(\a\oplus\c\oplus\b)$
the permutation $\sig'$ defined by
$\sig'_1,\ldots,\sig'_{r-1}=\sig_1,\ldots,\sig_{r-1}$,
$\sig'_{r+n+1},\ldots,\sig'_{r+n+s}=\sig_{r+n+1},\ldots,\sig_{r+n+s}$,
and $\sig'_{r}\sig'_{r+1}\cdots\sig'_{r+n}$ is the unique
permutation of the set $\{\sig_{r},\sig_{r+1},\ldots,\sig_{r+n}\}$
which is order isomorphic to $\phi_{\tau_1,\tau_{n+1}}(\tau)$, where
$\tau$ is the permutation of $[n+1]$ which is order isomorphic to
$\sig_{r}\sig_{r+1}\cdots\sig_{r+n}$. It is easily checked that
$\Gamma$ is a well-defined function from $\P(\a\oplus\c\oplus\b)$ to
$P(\a\oplus\c'\oplus\b)$ which is injective but, by
Lemma~\ref{lem:comparison}, not surjective.

 The proof of (2) is  easier than and very similar to the proof of (1), and so is
left to the reader.
 \qed

A simple but important consequence of the above result is given in
the next section.

\section{The largest part in a maximal composition}

\begin{prop}\label{prop:part6}
Let $\c=(c_1,\ldots,c_k)$ be a composition of $n$. If $k=1$ and $n\geq 5$, then $\c$ is not maximal.
If $k\geq 2$ and
there exists $i$, $1\leq i\leq k-1$, such that $c_i\geq 5$ or
$c_k\geq 4$, then $\c$ is not maximal.
\end{prop}


The above result is an immediate consequence of
Proposition~\ref{prop:comparison2} and the two following lemmas.

\begin{lem}\label{lem:parts>4-a}
For $n\geq 3$, we have
 \begin{itemize}
 \item[(1a)] $\Int_{n,a}(n)=\Int_{a,n}(n)=2^{a-2}$ if $2\leq a\leq n-1$,
 \item[(1b)] $\Int_{a,b}(n)=0$ if $a\neq n$ and $b\neq n$, or $a=1$, or
 $b=1$;\\[-0.2cm]
 \item[(2a)] $\Ini_{1,n}(n)=1$ and $\Ini_{1,b}(n)=0$ if $b< n$,
 \item[(2b)] $\Ini_{a,b}(n)=\Int_{a,b}(n)$ if $a\neq 1$.
\end{itemize}
\end{lem}

\begin{proof} (1) Suppose $2\leq a\leq n-1$. A permutation $\sigma$ in $\S_n$
is in $Int_{n,a}(n)$ if and only it writes
$\sigma=n\,w_1\,1\,w_2\,a$ where $w_1$ is decreasing and $w_2$ is
increasing. There are $2^{a-2}$ choices for $w_1$ and each choice of
$w_1$ determines uniquely $w_2$. The second assertion is left to the
reader.

(2a) Clearly, for $n\geq 2$, $Ini_{1,n}(n)=\{1\,2\,\cdots\,n\}$. If
$b<n$ and $\sigma\in\S_n$ with $\sigma_1=1$ and $\sigma_n=b$, then
$\sigma$ has at least one peak (the index $i$ such that $\sig_i=n$)
and thus $\pc(\sig)\neq (n)$, whence $Ini_{1,b}(n)=\emptyset$.

(2b) By definition, for all integers $a,b$, we have
$Int_{a,b}(n)\subseteq Ini_{a,b}(n)$. If $b=1$,
$Ini_{a,b}(n)=\emptyset$. Suppose $a\neq 1$, $b\neq 1$ and $\sig\in
Ini_{a,b}(n)$. If $\sig\notin Ini_{a,b}(n)$ (i.e., if
$\sig_1<\sig_2$), the greatest integer $p$ such that
$\sig_1<\sig_2<\cdots<\sigma_p$ is less than $n-1$ since $\sig_i=1$
for some index $i$ such that $2\leq i\leq n-1$. This implies that
$p$ is a peak of $\sig$ and thus $\pc(\sig)\neq (n)$. Consequently,
$Ini_{a,b}(n)\subseteq Int_{a,b}(n)$.
 \end{proof}

\begin{lem}\label{lem:parts>4-b}
For $n\geq 5$, we have
\begin{itemize}
\item[(1a)] $\Int_{n,n-1}(3,n-3)=\Int_{n-1,n}(3,n-3)=(n-4)2^{n-4}$,
\item[(1b)] $\Int_{n,a}(3,n-3)=(n-2-a)2^{a-2}+\chi(a\geq3)(a-1)2^{a-3}$ for
$2\leq a\leq n-2$,
\item[(1c)] $\Int_{a,n}(3,n-3)=\Int_{a,n-1}(3,n-3)+(a-1)2^{n-5}$ for
$2\leq a\leq n-2$;\\[-0.2cm]
\item[(2)] $\Ini_{1,n}(3,n-3)>1$.
\end{itemize}
\end{lem}

\begin{proof}
 (1a) Clearly, $\Int_{n,n-1}(3,n-3)=\Int_{n-1,n}(3,n-3)$.
A permutation $\sigma$ in $\S_n$ is in $Int_{n,n-1}(3,n-3)$ if and
only it writes $\sigma=n\,x\,k\,w_1\,\ell\,w_2\,n-1$ where $w_1$ is
decreasing and $w_2$ is increasing, $1 \leq x<k$, $1\leq \ell\leq
k\leq n-2$ and all the letters of $w_1$ and $w_2$ are greater than
$\ell$. It is easily checked that for each $k$, there are
$(k-1)2^{k-3}$ such permutations. Thus,
$\Int_{n,n-1}(3,n-3)=\sum_{k=3}^{n-2}(k-1)2^{k-3}=(n-4)2^{n-4}$.

 (1b) Suppose
$2\leq a\leq n-2$. Clearly, if $\sig\in Int_{n,a}(3,n-3)$, then we
have $\sig_3=n-1$. We classify the permutations $\sigma$ in
$Int_{n,n-1}(3,n-3)$ according to the value of $\sig_2$. If $a<k\leq
n-2$, then there are $2^{a-2}$ choices for the permutations $\sig$
in $Int_{n,a}(3,n-3)$ with $\sig_2=k$. If $1\leq k<a$ and $a\geq 3$
(resp., $a=2$), then there are $2^{a-3}$ (resp., 0) choices.

(1c) Suppose $2\leq a\leq n-2$. We classify the permutations
$\sigma$ in $Int_{a,n}(3,n-3)$ according to the value of
$\sig_{n-1}$. If $\sig_{n-1}=n-1$, then there are
$\Int_{a,n-1}(3,n-3)$ such permutations. If $\sig_{n-1}\neq n-1$,
then $\sig_3=n-1$ and there are $a-1$ choices for $\sig_2$. For each
choice of $\sig_2$, there are $2^{n-5}$ corresponding permutations.

(2) It suffices to observe that $\{1\,2\,4\,3\,5\,6\,\cdots n\,,\,
1\,3\,4\,2\,5\,6\,\cdots n\}\subseteq Ini_{1,n}(3,n-3)$.
 \end{proof}

\noindent{\bf \emph{Proof of Proposition~\ref{prop:part6}.}} (1)
Suppose first $k\geq 2$. It follows from Lemmas~\ref{lem:parts>4-a}
and~\ref{lem:parts>4-b} that
\begin{itemize}
 \item for $n\geq 6$ and all $a,b\in[n]$,
we have $\Int_{a,b}(n)\leq \Int_{a,b}(3,n-3)$,  and $\Int_{2,n}(n)<
\Int_{2,n}(3,n-3)$,
 \item for $n\geq 5$ and all $b\in[n]$, $\Ini_{\cdot,b}(n)\leq \Ini_{\cdot,b}(3,n-3)$, and
$\Ini_{\cdot,n}(n)<\Ini_{\cdot,n}(3,n-3)$.
\end{itemize}
Consequently, by Proposition~\ref{prop:comparison2}, if $n\geq 5$, then for any nonempty compositions $\a$
and $\b$,  the compositions  $\a\oplus (n)\oplus
\b$ and $(n)\oplus \b$ are not maximal. In other words, if $\c$ is maximal then $c_i\leq 4$ for $i=1,\ldots,k-1$.
This, combined with Fact~\ref{fact:rev-compo}, implies that if $\c$ is maximal, $c_k\leq 3$.

(2) If $k=1$, then a simple counting (see the proof of the above
lemmas or~\cite{Billey}) shows that $\Pa(n)=2^{n-1}$ and
$\Pa(3,n-3)=(\binom{n-1}{2}-1)2^{n-2}$, from which we deduce that $\Pa(n)<\Pa(3,n-3)$ for
$n\geq 5$.
 \qed

\section{Some forbidden patterns in a maximal composition}

The $3$-factorization of a composition $\c$ is the (unique)
factorization of $\c$ as
$$
\c=\c^{(0)}\oplus(3)\oplus
\c^{(1)}\oplus\cdots\oplus(3)\oplus\c^{(k)}
$$
where $k$ is the number of parts in $\c$ equal to 3  and the
compositions $\c^{(i)}$ (possibly empty), called factors, have no part equal to 3. For instance,
$(4,4,3,2,4,2,3,3,2,1)=(4,4)\oplus(3)\oplus(2,4,2)\oplus(3)\oplus(3)\oplus(2,1)$.

The purpose of this section is to prove the following result.
\begin{prop}\label{prop:canform1}
Suppose $\c$ is a maximal composition with at least 3 parts. Then,
\begin{align}\label{eq:canform1}
 \c&=x_0\oplus(3)\oplus x_1\oplus\cdots\oplus(3)\oplus x_k
\end{align}
for some $k\geq 1$ and some sequence of compositions $(x_i)_i$ such
that\newline \centerline{$x_{0}\in\{\ep,(4)\}$, $x_{k}\in\{\ep\}\cup\{(2^s): s\geq1\}$,
 and, for $i=1,\ldots,k-1$, $x_{i}\in\{\ep\}\cup\{(2^s),(4^t): s,t\geq1\}$.}
\end{prop}

Clearly, in view of Proposition~\ref{prop:part6} and
Fact~\ref{fact:admissible}, the above result is immediate from the
following lemma.

\begin{lem}\label{lem:ForbiddenPatternsIni}
For any nonempty compositions $\a$ and $\b$, 
\begin{enumerate}
 \item if $\c\in\{(2,2),(2,3),(2,4)\}$, the composition
 $\c\oplus\b$ is not maximal;
 \item if $\c\in\{(2,1),(3,1),(4,1)\}$,
 the composition $\a\oplus\c$ is not maximal;
 \item if $\c\in\{(2,4),(4,2)\}$,
 the compositions $\c\oplus\b$, $\a\oplus\c$ and $\a\oplus\c\oplus\b$ are not maximal;
  \item the composition $(4,4)\oplus\b$ is not maximal.
\end{enumerate}
\end{lem}

\begin{proof} For simplicity, if $\c$ is a composition of $n$,
the sequence $\big(\Ini_{\cdot,b}(\c)\big)_{1\leq b\leq n}$ will be denoted by $\T\,\c$.
Using a computer algebra system, it is easy to obtain
\begin{align*}
\T(2,2)&=[0,1,2,2],\quad \T(4)=[0,1,2,4], \quad
\T(2,3)=[0,2,4,6,8],\quad \T(3, 2)=[0,3,6,8,8],\\
\T(2,4)&=[0,3,6,10,16,24],\quad \T(4,2) =[0,7,14,20,24,24],\quad
\T(3,3)=[0,8,16,24,32,40],\\
\T(3,4)&= [0,15,30,48,72,104,144],\quad
\T(4,3)= [0,24,48,72,96,120,144],\\
\T(4,4)&=[0,55,110,172,248,344,464,608],\quad
\T(3,2,3)=[0,96,192,288,384,480,576,672].
\end{align*}
from which it results that
\begin{align*}
 \T(2,2)<\T(4),\;\;
 \T(2,3)<\T(3,2),\;\;
 \T(2,4)<\T(4,2)<\T(3,3),\;\;
 \T(3,4)<\T(4,3),\;\;
 \T(4,4)<\T(3,2,3).
\end{align*}
By Proposition~\ref{prop:comparison2}(2), these inequalities imply
that if $\c\in\{(2,2),(2,3),(2,4),(4,2),(3,4),(4,4)\}$, the composition $\c\oplus\b$ is not maximal
for any $\b\neq\ep$. By Fact~\ref{fact:rev-compo}, this can be reformulated as follows: if $\c\in\{(2,1),(3,1),(4,1),(2,3),(4,2),(4,3)\}$,
the composition $\a\oplus\c$  is not maximal for any $\a\neq\ep$.

Thus, to conclude our proof, it just remains to check that if $\c\in\{(2,4),(4,2)\}$,
 the composition $\a\oplus\c\oplus\b$ is not maximal for any $\a\neq\ep$ and any $\b\neq\ep$.
\begin{table}[h!]
\begin{center}\begin{tabular}{c|rrrrrr}
$a\setminus b$&2&3&4&5&6&7\\
\hline\\[-0.3cm]
 2& [ 0, 0]& [0, 0]& [ 1, 2]& [ 2, 4]&
[3, 6]&  [3, 8]\\
 3& [ 0, 0]& [ 0, 0]& [ 2, 4]& [ 4, 8]& [6, 12]&  [ 6, 16]\\
 4& [ 1, 2]& [ 2, 4]& [0, 0]& [ 6,12]&
  [ 10, 18]& [ 10, 24]\\
 5& [ 2, 4]& [ 4, 8]& [ 6, 12]& [ 0, 0]&
  [ 16, 24]& [ 16, 32]\\
 6& [ 4, 6]& [ 8, 12]& [ 12, 18]& [ 16,24]&  [ 0, 0]& [ 24, 40]\\
 7& [ 7, 8]& [ 14, 16]& [ 20, 24]& [ 24,32]&  [ 24, 40]& [ 0, 0]\\
\end{tabular}
\end{center}
\caption{The values
$[\Int_{a,b}(5,2),\Int_{a,b}(4,3)]$}
\label{tab:42vs33}
\end{table} 
Again, by using a computer algebra system, we can explicitly compute the array
$\big([\Int_{a,b}(5,2),Int_{a,b}(4,3)]\big)_{1\leq a,b\leq 7}$ (see Table~\ref{tab:42vs33})
from which it results that for all integers $a,b$ we have
$\Int_{a,b}(5,2)\leq \Int_{a,b}(4,3)$ and (at least) one of these inequalities is strict.
By Proposition~\ref{prop:comparison2}(1), this implies that $\a\oplus(4,2)\oplus\b$ is not maximal
for any $\a\neq\ep$ and any $\b\neq\ep$. By Fact~\ref{fact:rev-compo}, the same is true with $(4,2)$ replaced by $(2,4)$.
 \end{proof}

Before we end this section, we provide further patterns
 that a maximal composition must avoid.
 \begin{lem}\label{lem:ForbiddenPatternsIni2}
For any nonempty compositions $\a$ and $\b$, 
\begin{enumerate}
 \item if $\c\in\{(4,3,2),(4,3,4),(3,2,3,2),(3,3,2,3,2)\}$, the composition
 $\c\oplus\b$ is not maximal;
 \item if $\c\in\{(2,3,3),(4,3,3),(2,3,2,2),(2,3,2,3,2)\}$,
 the composition $\a\oplus\c$ is not maximal.
 \end{enumerate}
\end{lem}

\begin{proof}
By Fact~\ref{fact:rev-compo}, the two assertions are equivalent. So it suffices to prove (1).
Using a computer algebra system, we obtain
\begin{align*}
\T(4,3,2)&= [0, 504, 1008, 1488, 1920, 2280, 2544, 2688, 2688],\\
\T(3,3,3)&= [0, 560, 1120, 1680, 2240, 2800, 3360, 3920, 4480],\\
\T(4,3,4)&= [0, 9072,18144, 27720, 38304, 50376, 64368,80640,99456,120960, 145152],\\
\T(3,2,3,3)&=[0,16128,32256,48384,64512,80640,96768,112896,129024,145152,161280],\\
\T(3,2,3,2)&=[0,2688,5376,7968,10368,12480,14208,15456,16128,16128],\\
\T(4,3,3)&=[0,2688,5376,8064,10752,13440,16128,18816,21504,24192],\\
\T(3,3,2,3,2)&= [0, 887040, 1774080, 2644992, 3483648, 4273920,
4999680, 5644800, 6193152, \\&\quad 6628608, 6935040, 7096320, 7096320],\\
\T(4,3,3,3)&= [0,887040,1774080,2661120,3548160,4435200,
5322240,6209280,7096320,\\&\quad 7983360,8870400,9757440,10644480],
\end{align*}
from which it results that
\begin{align*}
 T(4,3,2)&<T(3,3,3),&
 T(4,3,4)&<T(3,2,3,3),&
 T(3,2,3,2)&<T(4,3,3),&
 T(3,3,2,3,2)&<T(4,3,3,3).
\end{align*}
This, by Proposition~\ref{prop:comparison2}(2), gives the first assertion.
\end{proof}

We can even go further with the same method but it will be more convenient
and even more simpler to use the tools  presented in  the next section.

\section{A counting lemma}

 The purpose of this section is to provide an efficient counting formula for the number $\Pa(\c)$
when $\c$ is a composition in which the factors of its 3-factorization have a small size. In the sequel, 
we let $|\c|$ denote the size of the composition $\c$. The size of the empty composition is given by $|\ep|=0$.
The following is the key result in this section.

\begin{lem}\label{lem:separation}
Let $\a$ and $\b$ two compositions. If $\b\neq \ep$, then we have
\begin{align}\label{eq:separation}
\Pa(\a\oplus(3)\oplus\b)&=\binom{|\a|+|\b|+3}{|\a|+1}\Pa(\a\oplus(1))\Pa((2)\oplus\b).
\end{align}
\end{lem}

\begin{proof}  Given a composition $\c$ of a positive integer $n$ 
and a finite subset $F\subset\mathbb{P}$ with cardinality $\# F=n$, 
we let $\P(\c\,;F)$ be the set of permutations of the set $F$ having peak set equal to $F$.
With this terminology, we have $\P(\c)=\P(\c\,;[n])$. 

Let $r$ and $s$ be the sizes of the compositions $\a$ and $\b$,
and let
$$R_{\a,\b}=\{(\gamma,\beta)\,:\gamma\in\P(\a\oplus(1)\,;S),\;
\beta\in\P((2)\oplus\b\,;T),\;S\cup T=\{1,\ldots,r+s+3\}\}.$$ 
Note that $\#R_{\a,\b}$ is equal to the right-hand side of~\eqref{eq:separation}. 
 So, to prove~\eqref{eq:separation}, it suffices to present a bijection between $R_{\a,\b}$ 
 and $\P(\a\oplus(3)\oplus\b)$. This is quite easy. 
 
 Consider the concatenation function
$\Gamma:R_{\a,\b}\mapsto\S_{r+s+3}$ defined by
$\Gamma(\gamma,\beta)=\gamma\beta$. It is clear that $\Gamma$ is
well defined, injective and $\Gamma(R_{\a,\b})\supseteq
\P(\a\oplus(3)\oplus\b)$. We now check that
$\Gamma(R_{\a,\b})\subseteq \P(\a\oplus(3)\oplus\b)$. Suppose
$\sig=\gamma\beta$ with $(\gamma\beta)\in R_{\a,\b}$. To prove that
$\sig=\gamma\beta$ is in $\P(\a\oplus(3)\oplus\b)$, it suffices to
prove that $r+1$ and $r+2$ are not peaks of $\sig$ which is immediate
since $\sig_{r}=\gamma_r>\gamma_{r+1}=\sig_{r+1}$ ($r$ is a peak of
$\gamma$) and $\sig_{r+2}=\beta_{1}<\beta_{2}=\sig_{r+3}$ ($2$ is a
peak of~$\beta$).
 \end{proof}

Repeated applications of Lemma~\ref{lem:separation}  leads to
the following result the proof of which is omitted.

\begin{prop}\label{prop:separation}
 Let $\big(\c^{(i)}\big)_{0\leq i\leq k}$ be a sequence of compositions
such that $\c^{(k)}\neq \ep$. Then,  we have
\begin{align}\label{eq:counting}
\Pa(\c^{(0)}\oplus(3)\oplus\c^{(1)}\oplus(3)\cdots\oplus\c^{(k)})
&=\binom{\sum_{i=0}^{k}|\c^{(i)}|\,+3k}{\c^{(0)}+1,\,3+|\c^{(1)}|,\ldots,3+|\c^{(k-1)}|,\,\c^{(k)}+2}\\
&\quad\times \Pa(\c^{(0)}\oplus(1))\left(\prod_{i=1}^{k-1}
\Pa((2)\oplus\c^{(i)}\oplus(1))\right)\Pa((2)\oplus\c^{(k)}).\nonumber
\end{align}
\end{prop}
In the sequel, we will use the  following formulas:
\begin{align}\label{eq:ExactPc}
&\Pa(2,1)=2,\quad \Pa(3,1)=8,\quad \Pa(4,1)=24,\quad \Pa(2,2,1)=16,\nonumber\\
&\Pa(2,2)=8,\quad\Pa(2,3)=24,\quad\Pa(2,2,2)=96,\quad \Pa(2,4)=64,\\
&\Pa(2,2,1)=16,\quad \Pa(2,2,2,1)=272,\quad\Pa(2,4,1)=288.\nonumber
\end{align}

The following non maximality criterion is immediate from the above result.
  \begin{prop}\label{prop:comparison-factors}
Suppose $\c=\c^{(0)}\oplus(3)\oplus\c^{(1)}\oplus(3)\cdots\oplus\c^{(k)}$
with $k\geq 1$ and $\c^{(k)}\neq\ep$. If there exists a composition $\b$ such that
\begin{enumerate}
 \item $|\b|=|\c^{(i)}|$ for some index $i$, $1\leq i\leq k-1$, and $\Pa((2)\oplus\c^{(i)}\oplus(1))<\Pa((2)\oplus\b\oplus(1))$, 
 then $\c$ is not maximal;
 \item $|\b|=|\c^{(k)}|$  and $\Pa((2)\oplus\c^{(k)})<\Pa((2)\oplus\b)$, 
 then $\c$ is not maximal.
\end{enumerate}
\end{prop}

\begin{cor}\label{cor:canform2}
Suppose $\c$ is a maximal composition with at least 3 parts and with 3-factorization
\begin{align}\label{eq:canform}
 \c&=x_0\oplus(3)\oplus x_1\oplus\cdots\oplus(3)\oplus x_k.
\end{align}
Then, $k\geq 1$ and the sequence of compositions $(x_i)_i$ satisfies\newline
\centerline{$x_{0}\in\{\ep,(4)\}$, $x_{k}\in\{\ep,(2),(2,2)\}$, and $x_{i}\in\{\ep,(2),(4)\}$ for $i=1,\ldots,k-1$.}
\end{cor}

\begin{proof} By Proposition~\ref{prop:canform1}, we can assume that 
$ \c=x_0\oplus(3)\oplus x_1\oplus\cdots\oplus(3)\oplus x_k$
with $k\geq 1$ and $x_{0}\in\{\ep,(4)\}$, $x_{k}\in\{\ep\}\cup\{(2^s): s\geq1\}$,
 and $x_{i}\in\{\ep\}\cup\{(2^s),(4^t): s,t\geq1\}$ for $i=1,\ldots,k-1$.
 
 Using a computer algebra system, we obtain
\begin{align*}
\T(2,2,2,2)&= [0, 61, 122, 178, 224, 256, 272, 272],\quad
\T(2, 3, 3) = [0, 80, 160, 240, 320, 400, 480, 560]\\
\T(2, 4, 4) &= [0, 889, 1778, 2726, 3792, 5032, 6496, 8224, 10240,12544],\\
\T(2, 3, 2, 3) &= [0, 1792, 3584, 5376, 7168, 8960, 10752, 12544,14336, 16128],
\end{align*}
from which it results that  $\T((2)\oplus(2,2,2))<\T((2)\oplus(3,3))$
and $\T((2)\oplus(4,4))<\T((2)\oplus(3,2,3))$. 
Moreover, it is easily checked that $\Pa((2)\oplus(2,2,2))<\Pa((2)\oplus(3,3))$.
These inequalities combined with Propositions~\ref{prop:comparison2}(2) and~\ref{prop:comparison-factors}
imply that $x_{k}\in\{\ep,(2),(2,2)\}$, and $x_{i}\in\{\ep,(2),(2,2),(4)\}$ for $i=1,\ldots,k-1$.

 To complete our proof, it remains to prove that if $\c=\a\oplus(3)\oplus(2,2)\oplus(3)\oplus\b$,
 then $\c$ is not maximal. Suppose $\b=\ep$, then $r' \c=(4,2)\oplus \u$ for some composition $u\neq \ep$. Then,
 by Lemma~\ref{lem:ForbiddenPatternsIni} and Fact~\ref{fact:rev-compo}, $r'\c$ and $\c$ are not maximal.
  If $\b\neq\ep$, it results
 from~\eqref{prop:separation} and~\eqref{eq:ExactPc} that
\begin{align}
\frac{\Pa(\a\oplus(3)\oplus(2,2)\oplus(3)\oplus\b)}{\Pa(\a\oplus(3)\oplus(4)\oplus(3)\oplus\b)}
 &=\frac{\Pa((2)\oplus(2,2)\oplus (1))}{\Pa((2)\oplus(4)\oplus
 (1))}=\frac{272}{288}<1.
\end{align}
This concludes the proof.
\end{proof}

Before we end this section, we prove Theorem~\ref{thm:Maximalclasses}. 
It is convenient to present a slightly more general result 
than Proposition~\ref{prop:comparison-factors} which is easily derived by an appropriate specialization 
in Proposition~\ref{prop:separation} and use of the relation $\Pa(2,1)=2$.
 \begin{prop}\label{prop:counting-general}
For $k\geq 1$, $\ell_1,\ldots,\ell_k\geq 1$ and $\c^{(k+1)}\neq
\ep$, we have
\begin{align}
&\Pa(\c^{(1)}\oplus(3^{\ell_1})\oplus\c^{(2)}\oplus(3^{\ell_2})\oplus\cdots\oplus\c^{(k)}\oplus(3^{\ell_k})\oplus\c^{(k+1)})\nonumber\\
&\quad=\left(\frac{1}{3}\right)^{\sum_{i=1}^k\ell_i\,-k}\frac{\big(3(\ell_1+\cdots+\ell_k)+\sum_{i=1}^{k+1}|c_i|\big)!}
            {(1+|\c^{(1)}|)!\big(\prod_{i=2}^k(3+|\c^{(i)}|)!\big)\cdots!(|\c^{(k+1)}|+2)!}\label{eq:counting-general}\\
            &\quad\quad\times
            \Pa(\a\oplus(1))\left(\prod_{i=1}^k\Pa((2)\oplus\c^{(i)}\oplus(1))\right)\Pa((2)\oplus\c^{(k+1)}).\nonumber
\end{align}
\end{prop}
 
\begin{cor}\label{cor:ExactPc}
For $\ell\geq2$, $s,m\geq1$, $t\geq 0$, we have
\begin{align}
\Pa(3^{\ell})&=\Pa(4,3^{\ell-1},2)=\tfrac{1}{5}3^{2-\ell}(3\ell)!,\label{eq:P-T3ell}\\
\Pa(3^{\ell},2)&=3^{-\ell}(3\ell+2)!,\label{eq:P-3ell2}\\
\Pa(3^{s},2,3^{m})&=\Pa(4,3^{m-1},2,3^{s-1},2)=\tfrac{2}{25} 3^{-s-m}(3s+3m+2)!,\label{eq:P-3s23t} \\
\Pa(3^s,2,3^t,2)& =\Pa(3^{t+1},2,3^{s-1},2)=\tfrac{2}{5} 3^{-s-t}(3s+3t+4)!,\label{eq:P-3s23t2}\\
\Pa(4,3^{\ell})&=\tfrac{1}{25}3^{2-\ell}(3\ell+4)!.\label{eq:P-43s}
\end{align}
\end{cor}

\begin{proof} Specializing~\eqref{eq:counting-general} at $k=1$,
$\c^{(1)}=\ep$, $\c^{(2)}=(3)$ and $\ell_1=\ell-1$ gives
\begin{align*}
\Pa(3^{\ell})&=\left(\frac{1}{3}\right)^{\ell-2}\frac{(3\ell)!}{1!5!}\Pa(\ep\oplus(1))\Pa((2)\oplus(3)),
\end{align*}
which simplifies to~\eqref{eq:P-T3ell}  by using~\eqref{eq:ExactPc}.
Similarly, specializing~\eqref{eq:counting-general} at $k=2$,
$\c^{(1)}=\ep$, $\c^{(2)}=\c^{(3)}=(2)$, $\ell_1=s$ and $\ell_2=t$
gives
\begin{align*}
\Pa(3^s,2,3^t,2)&=\left(\frac{1}{3}\right)^{s+t-2}\frac{(3s+3t+4)!}{1!5!4!}\Pa(\ep\oplus(1))\Pa((2)\oplus(2)\oplus(1))\Pa((2)\oplus(3)),
\end{align*}
which simplifies to~\eqref{eq:P-3s23t2}  by
using~\eqref{eq:ExactPc}. The proof of the other assertions are left
to the reader.
 \end{proof}


\section{Proof of Theorem~\ref{thm:main}}

 The purpose of this section is to complete the proof of Theorem~\ref{thm:main}.
We shall use the following result which is a direct consequence of Proposition~\ref{prop:separation}.
\begin{prop}\label{prop:invariance}
 Let $\big(\c^{(i)}\big)_{0\leq i\leq k}$ be a sequence of compositions
 such that $\c^{(k)}\neq \ep$. Then,
for any permutation $\sig\in\S_{k-1}$, we have
\begin{align}\label{eq:invariance}
&\Pa(\c^{(0)}\oplus(3)\oplus\c^{(1)}\oplus(3)\oplus\cdots\oplus\c^{(k)})
=\Pa(\c^{(0)}\oplus(3)\oplus\c^{(\sig_1)}\oplus(3)\oplus\cdots\oplus\c^{(\sig_{k-1})}\oplus(3)\oplus\c^{(k)}).
\end{align}
\end{prop}

Combining Corollary~\ref{cor:canform2} with the above result leads to the 
following result.
\begin{cor}\label{cor:main}
Let $\c=(c_1,\ldots,c_k)$  be a maximal composition with~$k\geq 3$.
\begin{enumerate}
\item If $c_1=3$, then $\c$ has no part equal to 4.
\item If $c_1=4$, then $\c$ has only one part equal to 4.
\item If $c_1=c_k=3$, then $\c$ has at most one part equal to $2$.
\item If $c_1=3$ and $c_k=2$, then $\c$ has at most two part equals to $2$.
\item If $c_1=4$, and $\c$ has a part equal to 2 then $\c=(4,3^s,2^t)$ for some $s,t\geq 1$, $t\leq 2$.
\end{enumerate}
\end{cor}

\begin{proof} By Corollary~\ref{cor:canform2}, we can suppose that
 $\c=x_0\oplus(3)\oplus x_1\oplus\cdots\oplus(3)\oplus x_s$
for some $s\geq 1$ and some sequence of compositions $(x_i)_i$ such
that $x_{0}\in\{\ep,(4)\}$, $x_{s}\in\{\ep,(2),(2,2)\}$ and $x_{i}\in\{\ep,(2),(4)\}$ for $i=1,\ldots,s-1$.

 (1) Suppose $c_1=3$  and $c_i=4$ for some $i\geq2$ (i.e., $x_0=\ep$ and $x_j=(4)$ for some $j$ with $1\leq j< s$).
 By~\eqref{eq:invariance}, we can assume that $x_1=(4)$ (i.e., $c_2=4$). In this case, $r'\c=(4,2)\oplus \u$
 for $\u\neq \ep$ and thus, $\c$ is not maximal by Lemma~\ref{lem:ForbiddenPatternsIni}. This contradicts our assumption.
 Thus $c_i\neq 4$ for $2\leq i\leq k$.

 (2) Suppose $c_1=4$ and $c_i=4$ for some $i\geq2$ (i.e., $x_0=x_j=(4)$ for some $j$ with $1\leq j< s$).
  By~\eqref{eq:invariance}, we can assume that $x_1=(4)$ (i.e., $c_3=4$). In this case, $\c=(4,3,4)\oplus\b$ for some $\b\neq\ep$,
  and thus, $\c$ is not maximal by Lemma~\ref{lem:ForbiddenPatternsIni2}.

  (3) Suppose $c_1=c_k=3$ and $\c$ has at least two parts equal to 2.
  By~\eqref{eq:invariance}, we can assume that $x_1=x_2=(2)$. In this case, $\c=(3,2,3,2)\oplus\b$ for some $\b\neq\ep$,
  and thus, $\c$ is not maximal by Lemma~\ref{lem:ForbiddenPatternsIni2}.

 (4) Suppose $c_1=3$, $c_k=2$ and $\c$ has at least three parts equal to 2.
 There are two cases: $x_s=(2,2)$ and $x_s=(2)$.
 If $x_s=(2,2)$ (resp., $x_s=(2)$), by~\eqref{eq:invariance}, we can assume that $x_{s-1}=(2)$ (resp., $x_{s-2}=x_{s-1}=(2)$).
 In this case, $\c=\a\oplus(2,3,2,2)$ (resp., $\c=\a\oplus(2,3,2,3,2)$) for some $\a\neq\ep$.
  In both cases, $\c$ is not maximal by Lemma~\ref{lem:ForbiddenPatternsIni2}.

   (5) Suppose $c_1=4$ and $x_j=2$ for some $j$ with $1\leq j< s$. By~\eqref{eq:invariance},
 we can assume that $x_{1}=(2)$.   In this case, $\c=(4,3,2)\oplus\b$ for some $\b\neq\ep$,
  and thus, $\c$ is not maximal by Lemma~\ref{lem:ForbiddenPatternsIni2}.
 \end{proof}

Combining Corollaries~\ref{cor:canform2} and~\ref{cor:main} shows that 
\begin{align}\label{eq:reformulation-main}
&\textit{if $\c$ is a maximal composition with at least three parts, then}\\
&\quad\c\in\{(3^{\ell}), (4,3^{\ell}), (4,3^{\ell-2},2), (4,3^{\ell-2},2,2),
(3^{\ell},2), (3^{\ell},2,2), (3^{s},2,3^{t},2)\;:\;\ell\geq2,\; s\geq1,\;t\geq0\},\nonumber
\end{align}
which is very close to Theorem~\ref{thm:main}. 
On the other hand, using Fact~\ref{fact:rev-compo} and Corollary~\ref{cor:ExactPc}, we see that
for $s\geq1$, $t\geq0$, $k\geq1$ and $0\leq j\leq k-1$, $\ell\geq 2$, we have
\begin{align}\label{eq:comparisonmaxcomp}
&\Pa(4,3^{\ell-2},2)=\Pa(3^{\ell}),\quad \Pa(4,3^{s})<\Pa(3^{s},2,2),\quad  \Pa(3^{s},2,3^{t})<\Pa(3^{s+t},2),\nonumber\\
&\Pa(4,3^{\ell-2},2,2)<\Pa(3^{\ell},2),\quad \Pa(3^{k-j},2,3^{j},2)=\Pa(3^{k},2,2).
\end{align}

Combining~\eqref{eq:reformulation-main} with~\eqref{eq:comparisonmaxcomp} immediately implies that
Theorem~\ref{thm:main} is true for compositions with at least three parts.
The validity of Theorem~\ref{thm:main} for compositions with two parts 
(it suffices, by Proposition~\ref{prop:part6}, to consider compositions with parts less than 5)
can be treated by computer or by hand, and so, it is left to the reader.


\medskip
{\bf Acknowledgements.} The author would like to thank Christian Krattenthaler for fruitful discussions 
and its encouragement during the preparation of this paper.


\end{document}